\documentclass{article}
\usepackage{amsmath,amsfonts,amsthm}
\author{Jean Michel}
\date{8th July 2014}
\title{``Case-free'' derivation for Weyl groups of the number of
 reflection factorizations of a Coxeter element}
\newcommand\bB{{\mathbf B}}
\newcommand\bG{{\mathbf G}}
\newcommand\bT{{\mathbf T}}
\newcommand\bX{{\mathbf X}}
\newcommand\cE{{\mathcal E}}
\newcommand\cF{{\mathcal F}}
\newcommand\cH{{\mathcal H}}
\newcommand\cR{{\mathcal R}}
\newcommand\BF{{\mathbb F}}
\newcommand\BC{{\mathbb C}}
\newcommand\Fq{{\BF_q}}
\newcommand\GF{{\bG^F}}
\newcommand\Irr{{\hbox{Irr}}}
\newcommand\Fqbar{{\overline{\mathbb F}_q}}
\newcommand\Qlbar{{\overline{\mathbb Q}_\ell}}
\newcommand\inv{^{-1}}
\DeclareMathOperator\Fam{{\Xi}}
\DeclareMathOperator\GL{\text{GL}}
\DeclareMathOperator\Id{\text{Id}}
\DeclareMathOperator\Ind{\text{Ind}}
\DeclareMathOperator\Ref{\text{ref}}
\DeclareMathOperator\Trace{\text{Trace}}
\newcommand{\scal}[3]{{\langle\,#1,#2\,\rangle_{#3}}}
\newtheorem{lemma}[equation]{Lemma}
\theoremstyle{remark}
\newtheorem{remark}[equation]{Remark}

\begin{document}
\maketitle
Let   $W\subset\GL(\BC^n)$   be   an   irreducible  well-generated  complex
reflection  group, let $\cR$ be the set of its reflections, $\cR^*$ the set
of its reflecting hyperplanes, and let $c$ be a Coxeter element of $W$
(see \cite[Remark 1.3]{CT}).

In  \cite[Theorem 1.1]{CT}, Chapuy and Stump  obtain a very nice generating
series   for   the   number  $N_l:=|\{r_1,\ldots,r_l\in\cR^l\mid  r_1\ldots
r_l=c\}|$  of factorizations  of $c$  into the  product of  $l$ elements of
$\cR$.   Their   formula   is   $$\sum_{l\ge  0}  \frac{t^l}{l!}N_l=  \frac
1{|W|}(e^{t|\cR|/n}-e^{-t|\cR^*|/n})^n.$$

Their  method is to obtain a character-theoretic expression for $N_l$, that
they  proceed to evaluate case by case. My observation is that, in the case
of Weyl group, a uniform evaluation of their character-theoretic expression
can  be done using  properties of Deligne-Lusztig  representations. I thank
Christian  Stump for  making me  aware of  the problem,  and for  a careful
reading of this text.

We start with \cite[formula above (10)]{CT} which states that if $S$ is the
element  of  the  group  algebra  $\BC  W$ given by $\sum_{r\in\cR}r$, then
$|W|N_l=\sum_{\chi\in\Irr(W)}\chi(1)\chi(S^lc\inv)$. We observe that $S$ is
in the center of $\BC W$, thus the formula can also be written
$$|W|N_l=\sum_{\chi\in\Irr(W)}\chi(S^l)\chi(c\inv) \eqno(1).$$

For $\chi\in\Irr(W)$ write the fake degree of $\chi$ as
$x^{e_1}+\ldots+x^{e_{\chi(1)}}$  and define
$N(\chi^*)=\sum_{i=1}^{\chi(1)}  e_i$, where  $\chi^*$ denotes  the complex
conjugate (see for example \cite[\S 4.B]{BM}). 
We will need the following property:

\begin{lemma} \label{chi(S)} Assume that $|C_W(H)|$ has a common value $e$
for all $H\in\cR^*$. Then
$\chi(S)=|\cR|\chi(1)-N(\chi)-N(\chi^*)$.
\end{lemma}
\begin{proof}
For $H\in\cR^*$, whose fixator $C_W(H)$ is generated
by a reflection  $s_H$ of hyperplane $H$ and non-trivial eigenvalue
$\zeta=\exp(2i\pi/e)$, let $m_{H,i}(\chi)$ be the multiplicity of the eigenvalue
$\zeta^i$ of $s_H$ in the representation affording $\chi$. Then we get
$$\chi(\sum_{j=1}^{e-1}s_H^j)=
\sum_{i=0}^{e-1}m_{H,i}(\chi)\sum_{j=1}^{e-1}\zeta^{ij}
=(e-1)m_{H,0}(\chi)-\sum_{i=1}^{e-1}m_{H,i}(\chi).$$
Now the formulae in \cite[Corollaire 4.2]{BM} read, in our case where
$|C_W(H)|$ has a common value $e$:
$$\begin{aligned}
\sum_{H\in\cR^*}\sum_{i=1}^{e-1}m_{H,i}(\chi)&=\frac{N(\chi)+N(\chi^*)}e\\
\sum_{H\in\cR^*}m_{H,0}(\chi)&=\frac{|\cR|\chi(1)}{e-1}
-\frac{N(\chi)+N(\chi^*)}e
\end{aligned}$$
whence the Lemma.
\end{proof}
\begin{remark}
Another expression valid without assuming $|C_W(H)|$ constant is 
$\chi(S)=|\cR|\chi(1)-N(\chi)-N(\iota(\chi))$
where $\iota$ is {\em Opdam's involution}; to see this, differentiate with
respect to $x$ and then evaluate
at $x=1$ formula \cite[6.5]{Ma}; see also \cite[6.8]{Ma}.
\end{remark}
We  now restrict  to the  case where  $W$ is  the Weyl group of a connected
reductive  algebraic group $\bG$ over an  algebraic closure $\Fqbar$ of the
finite  field $\Fq$ with $q$ elements. We assume that $\bG$ is defined over
$\Fq$   and  denote  by   $F$  the  Frobenius   endomorphism  defining  the
corresponding  $\Fq$-structure. Let  $\bT$ be  an $F$-stable  maximal torus
lying  in  an  $F$-stable  Borel  subgroup  $\bB$. We may identify $W$ with
$N_\bG(\bT)$ and we assume $\bG$ split, which means that $F$ acts trivially
on  $W$. For $w\in  W$, let us  denote by $R_w$  the (virtual) character of
$\GF$  defined by the Deligne-Lusztig induction $R_{\bT_w}^\bG(\Id)$, where
$\bT_w$ is an $F$-stable maximal torus of type $w$ (with respect to $\bT$).
Here  $R_w$  is  a  $\Qlbar$-character,  for  some  prime number $\ell$ not
dividing  $q$, but we will consider it as a complex character by choosing a
suitable  embedding  $\Qlbar\hookrightarrow\BC$.  The  set  $\cE(\GF,1)$ of
constituents  of  the  various  $R_w$  is  called the set of {\em unipotent
characters} of $\GF$.

The character $R_1$ identifies to that of $\Ind_{\bB^F}^\GF\Id$, and, since
the  commuting algebra of  this representation is  the Hecke algebra $\cH$,
isomorphic to the group algebra of $W$, we have a decomposition of the form
$R_1=\sum_{\chi\in\Irr(W)}\chi(1)  U_\chi$ where  the $U_\chi$  are certain
characters of $\GF$ called the {\em principal series unipotent characters}.

The set $\cE(\GF,1)$ and the values $\scal{R_w}\rho\GF$ for
$\rho\in\cE(\GF,1)$  are  independent  of  $q$;  they provide an additional
combinatorial structure on $W$ which can actually be entirely determined by
the  Hecke  algebra.  In  the  case  where  $W$ is the symmetric group, the
characters $U_\chi$ exhaust the set $\cE(\GF,1)$ and
$\scal{R_w}{U_\chi}\GF=\chi(w)$,   so  Deligne-Lusztig  combinatorics  bring
nothing  new.  We  could  characterize  their  role in the current proof as
enabling  the use for other Weyl groups  of the same features which occur in
the proof in the symmetric group case.

By linearity we attach to any $a\in \BC W$ a class function $R_a$ on $\GF$,
given  if $a=\sum_{w\in W}a_w w$ by $R_a=\sum_w a_w R_w$. In the particular
case  where $a$ is the  idempotent $\frac 1{|W|} \sum_{w\in W}\chi(w\inv)w$
attached  to $\chi\in\Irr(W)$  we denote by $R_\chi$ the  corresponding class
function. Inverting we get $R_w=\sum_\chi\chi(w)R_\chi$. By
\cite[3.19.2]{Madison} we have
$\scal{R_\chi}{R_\psi}\GF=\delta_{\chi,\psi}$.  It follows that for any two
elements $a,b\in \BC W$ we have
$\scal{R_a}{R_b}\GF=\sum_{\chi\in\Irr(W)}\chi(a)\overline{\chi(b)}$.  Thus
formula (1) becomes $$|W|N_l=\scal{R_{S^l}}{R_c}.\eqno (1')$$

\begin{lemma} \label{hecke}
We have $\frac{\chi(S)}{\chi(1)}=|\cR|-a_\chi-A_\chi$, where $a_\chi$ 
(resp. $A_\chi$) is the valuation (resp. the degree) of the generic degree
of $\cH$ attached to $\chi$ (see for example \cite[\S 2.B]{BM}).
\end{lemma}
\begin{proof} This is just a translation of Lemma \ref{chi(S)} in terms of the
invariants coming from the Hecke algebra. In the case of Coxeter  groups 
in Lemma \ref{chi(S)} we have $e=2$ and  $\chi=\chi^*$.  Thus  Lemma
\ref{chi(S)} for Coxeter groups becomes $\chi(S)=\chi(1)|\cR|-2N(\chi)$. We
then conclude by \cite[formula 4.21]{BM} which states that in Coxeter groups
$N(\chi)=\chi(1)\frac{a_\chi+A_\chi}2$.
\end{proof}
\begin{lemma} \label{RSl} For $\rho\in\cE(\GF,1)$ we have
$$\scal{R_{S^l}}\rho\GF=\begin{cases}\chi(S^l)&\text{if $\rho=U_\chi$
for some $\chi\in\Irr(W)$}\\
0&\text{otherwise}\end{cases}.$$
\end{lemma}
\begin{proof} 
The blocks of the matrix
$\{\scal{R_\chi}{\rho}\GF\}_{\chi\in\Irr(W),\rho\in\cE(\GF,1)}$ are called the
{\em Lusztig families}. They constitute thus a partition $\Fam$ of $\cE(\GF,1)$ 
such that for a ``family'' $\cF\in\Fam$  we have:
\begin{itemize}\item
If $U_\chi\in\cF$, $\rho\notin\cF$ then $\scal{R_\chi}{\rho}\GF=0$.
\end{itemize}
Given a family $\cF$, the invariants
$a_\chi$ and $A_\chi$ take a constant value on the $\chi\in\Irr(W)$ such that
$U_\chi\in\cF$ (see, for example \cite[4.23, 5.25 and 5.27]{Book}), 
thus $\chi(S)/\chi(1)$ takes a constant value that we will 
denote $c_\cF$ on a family $\cF$.  Thus, for $\rho\in\cF$ we have:
$$\begin{aligned}
 \scal{R_{S^l}}\rho\GF&=\scal{\sum_{\chi\in\Irr(W)}\chi(S^l)R_\chi}\rho\GF=
 \sum_{\{\chi\mid U_\chi\in\cF\}}\chi(S^l)\scal{R_\chi}\rho\GF\\
&=c_\cF^l\sum_{\{\chi\mid U_\chi\in\cF\}}\chi(1)\scal{R_\chi}\rho\GF
 =c_\cF^l\scal{\sum_{\chi\in\Irr(W)}\chi(1)R_\chi}\rho\GF\\
 &=\begin{cases}
0&\text{unless $\rho$ is a $U_\chi$}\\
c_\cF^l\chi(1)=\chi(S^l)&\text{if $\rho=U_\chi$}\end{cases}
\end{aligned}$$
\end{proof}
\begin{lemma}\label{Rc} If $c$ is a Coxeter element and $\chi\in\Irr(W)$
then $\scal{R_c}{U_\chi}\GF=0$ unless $\chi$ is
an exterior power of the reflection character $\Ref$ of $W$.
Moreover $\scal{R_c}{U_{\wedge^i\Ref}}\GF=(-1)^i$.
\end{lemma}
\begin{proof}
Let  $\bX_c$ be  the Deligne-Lusztig  variety defining  $R_c$, so  that for
$g\in\GF$  we have  $R_c(g)=\sum_i(-1)^i\Trace(g\mid H_c^i(\bX_c,\Qlbar))$.
By  \cite[6.7 (ii)]{Lusztig}  applied with  $I=\emptyset$ and $R_0=\Id$, we
know   that   in   each   cohomology   group   $H_c^i(\bX_c,\Qlbar)$   with
$i=n,n+1,\ldots,2n$,   there  is  exactly  one  irreducible  representation
$U_{\chi_i^W}$  in the principal series and that it has multiplicity 1. The
fact   that  $\chi_i^W=\wedge^{2n-i}\Ref$   is  \cite[remark  7.8]{Lusztig}
applied  with $I=\emptyset$ and  $R_0=\Id$. In this  case the Hecke algebra
$\cH$ of loc. cit. is the same as our algebra $\cH$ and the $i$-th power of
the  reflection  representation  of  $\cH$  defines precisely the character
$U_{\wedge^i\Ref}$.
\end{proof}
Using Lemmas \ref{RSl} and \ref{Rc} to evaluate $(1')$ we get:
$$|W|N_l=\sum_{i=0}^n(-1)^i\wedge^i\Ref(S^l).\eqno(1'')$$
\begin{lemma}For any well-generated irreducible complex 
reflection group, formula $(1'')$ is equivalent to the Chapuy-Stump formula.
\end{lemma}
\begin{proof} Since the representations $\wedge^i\Ref$ are irreducible by a
result  of  Steinberg  (see  \cite[\S  2  ex. 3]{St}), formula $(1'')$ can be
written
$$|W|N_l=\sum_{i=0}^n(-1)^i(\wedge^i\Ref(S)/\wedge^i\Ref(1))^l\wedge^i\Ref(1).$$
Let   us   pick   a   reflecting   hyperplane  $H$,  let  $e=|C_W(H)|$  and
$\zeta=\exp(2i\pi/e)$  and let $s$ be  a reflection with eigenvalue $\zeta$
which generates $C_W(H)$. Let us compute
$\sum_{j=1}^{j=e-1}\wedge^i\Ref(s^j)$.    We    may    choose    a    basis
$e_1,\ldots,e_n$  of $\BC^n$ such that  $e_i\in H$ for $i=1,\ldots,n-1$ and
$s e_n=\zeta e_n$. A basis of $\wedge^i\BC^n$ consists of
$e_I=e_{a_1}\wedge\ldots\wedge  e_{a_i}$ where  $I=\{a_1,\ldots,a_i\}$ with
$a_1<\ldots<a_i$  runs over all subsets  of $\{1,\ldots,n\}$ of cardinality
$i$.   We   have   $\sum_{j=1}^{e-1}s^j(e_I)=\begin{cases}  (e-1)e_I&\text{if
$n\notin I$}\\-e_I&\text{otherwise}\end{cases}$ whence
$\sum_{j=1}^{j=e-1}\wedge^i\Ref(s^j)=(e-1){n\choose  i}-e{n-1\choose  i-1}$
and $\wedge^i\Ref(S)=|\cR|({n\choose i}-{n-1\choose
i-1})-|\cR^*|{n-1\choose i-1}$. We finally get
$\wedge^i\Ref(S)/\wedge^i\Ref(1)=|\cR|(1-\frac   in)-|\cR^*|\frac  in$  and
$$|W|N_l=\sum_{i=0}^n(-1)^i(|\cR|(1-\frac  in)-|\cR^*|\frac in)^l {n\choose
i},$$  which  is  exactly  what  one  gets  when expanding the Chapuy-Stump
formula.
\end{proof}

\noindent{\sl Note.}
I  believe that Lemmas \ref{RSl} and \ref{Rc} hold for any Spetsial complex
reflection  group,  with  an  appropriate  definition  of  a  formal set of
unipotent  characters  (see  \cite{BMM}).  I  checked  it by computer 
for  the  primitive  irreducible Spetsial complex reflection
groups using \cite{chevie}.

For  Spetsial groups whose  reflections have  order 2  the proof  of Lemma \ref{RSl}
remains  formally  valid  since  Lemma  \ref{hecke}  remains  true for such
groups,  thus also the fact that $\chi(S)/\chi(1)$ is constant on families.
When reflections do not all have order 2,
there are examples where $\chi(S)/\chi(1)$ is not constant
on families so the proof   of Lemma \ref{RSl} has to change.


\begin{thebibliography}{DM2}
\bibitem[Bou]{St} N.~Bourbaki, ``Groupes et alg\`ebres de Lie'', Chap 5,
Hermann 1969.
\bibitem[BM]{BM} M.~Brou\'e  et   J.~Michel, ``Sur  certains  \'el\'ements
r\'eguliers  des  groupes de  Weyl  et  les vari\'et\'es  de  Deligne--Lusztig
associ\'ees'',  {\sl Progress  in Mathematics  n$^o$ \bf  141} (Birkhauser
1996) 73--139.
\bibitem[BMM]{BMM} M.~Brou\'e, G.~Malle and J.~Michel, 
``Split Spetses for primitive reflection groups'',{\sl Ast\'erisque \bf 359}
(2014) 1--146.
\bibitem[CT]{CT} G.~Chapuy and C.~Stump,``Counting factorizations of Coxeter
elements into products of reflections'', {\tt arXiv:1211.2789}, to appear in
{\sl Journal of London Math. Soc.}
\bibitem[Lu1]{Lusztig} G.~Lusztig, ``Coxeter Orbits and Eigenspaces of 
Frobenius'', {\sl Inventiones \bf 38}(1976), 101--159.
\bibitem[Lu2]{Madison} G.~Lusztig, ``Representations of finite Chevalley
 groups'', {\sl  Regional Conf. Series in Math. \bf 39} (1978) AMS, 48p.
\bibitem[Lu3]{Book} G.~Lusztig,
``Characters of reductive groups over a finite field'',
{\sl Ann. Math. Studies \bf 107}, Princeton UP (1984) 384p
\bibitem[Ma]{Ma} G.~Malle,``On the rationality and fake degrees of characters
of cyclotomic algebras'', {\sl J. Math. Sci. Univ. Tokyo \bf 6}(1999)
647--677.
\bibitem[Mi]{chevie} J.~Michel,``The development version of the {\tt CHEVIE}
package of {\tt GAP3.}'', {\tt arXiv:1310.7905 [math.RT]}.
\end{thebibliography}
\end{document}